\documentclass[10pt]{article}
\usepackage{mathrsfs}
\usepackage{amsfonts}
\usepackage{amsthm,amsmath,amssymb,anysize}
\newtheorem{lemma}{Lemma}[section]
\newtheorem{theorem}[lemma]{Theorem}

\newtheorem{proposition}[lemma]{Proposition}

\usepackage[numbers,sort&compress]{natbib}
\marginsize{4.2cm}{4.2cm}{3.6cm}{3cm}
\numberwithin{equation}{section}

\title{\textsf{Generators of simple modular Lie superalgebras }}
\author{\textsc{Liming Tang$^{1}$ }\textsc{and}
    \textsc{Wende Liu$^{1,2}$}\footnote{Correspondence: \texttt{wendeliu@ustc.edu.cn}} \footnote{Supported by the Natural Science Foundation for Distinguished Young Scholars, HLJ Province (JC201004) and  the National Natural Science Foundation of China (11171055)}
  \\
  \small\textit{$^{1}$Department of Mathematics},
  \small\textit{Harbin Institute of Technology}\\
  \small\textit{Harbin 150006, China}\\
 \small\textit{$^{2}$School of Mathematical Sciences},
 \small \textit{Harbin Normal University} \\
 \small \textit{Harbin 150025, China}}
\date{ }
\begin{document}
\maketitle
\begin{quotation}
{\small\noindent \textbf{Abstract}:
 Let $X$ be one of the finite-dimensional simple
graded Lie superalgebras of Cartan type $W, S, H, K, HO, KO, SHO$ or
$SKO$ over an algebraically closed field of
 characteristic $p>3$. In this paper we prove that
  $X$ can be generated by one element except the ones of type $W,$ $HO$, $KO$  or
$SKO$ in certain exceptional cases, in which
 $X$ can be generated by two elements. As a subsidiary result, we also prove that
 certain classical Lie superalgebras or
 their relatives can be generated by one or two elements.

 \vspace{0.05cm} \noindent{\textbf{Keywords}}: Lie
superalgebras of Cartan type; generators

\vspace{0.05cm} \noindent \textbf{Mathematics Subject Classification
2000}: 17B05, 17B50, 17B70}
\end{quotation}
 \setcounter{section}{-1}
\section{Introduction}
Over a field of characteristic $p>3$, the classification problem for finite
dimensional simple Lie superalgebras   remains open up to now.
However, one has known  many simple modular Lie superalgebras including
  the $\mathbb{Z}$-graded simple modular Lie superalgebras
of Cartan type $W, S, H, K, HO, KO, SHO$ or
$SKO$ (for example, see
\cite{Fu-Zhang-Jiang,Liu-He,Liu-Zhang-Wang,Liu-Yuan-1,Zhang}), which
are analogues of the eight series of infinite-dimensional vectorial
Lie superalgebras over $\mathbb{C}$ (for example, see \cite{Kac2}). Roughly
speaking, these $\mathbb{Z}$-graded Lie superalgebras are
subalgebras of the full
 superderivation algebras  of the tensor products  of the divided power algebras and the
 exterior superalgebras and their $\mathbb{Z}$-nulls are classical Lie superalgebras or their relatives.
 As in the modular Lie algebra case, it is conceivable that these Lie superalgebras will play a key role in the final classification of simple modular Lie superalgebras.

In this paper we mainly prove that a simple Lie superalgebra in the eight series mentioned above can be generated by 1 element
except  the ones of type $W, HO, KO$ or $SKO$ in certain exceptional cases, in which $X$ can be generated by 2 elements. In the process we also prove that
  the $\mathbb{Z}$-nulls of the eight series of Lie superalgebras of Cartan type, which are classical Lie superalgebras or
their relatives, can be
 generated by 1 element except $\mathfrak{gl}(m,n)$ with $m-n\equiv0\pmod{p},$ ${\widetilde{\rm{P}}(m)}$ with $
m\not\equiv0\pmod{p}$ and ${\widetilde{\rm{P}}(m)}\oplus
\mathbb{F}I$ with $ m\not\equiv0\pmod{p}$ and   in those exceptional
cases, the nulls can be generated by 2 elements.

This study is originally
motivated by two  papers of Bois, which provide us with a
 considerable amount of information for modular Lie algebras. In 2009, Bois \cite{MBJ}
proved that any simple Lie algebra in  characteristics $p\neq 2,3$ can be
generated by  2 elements and moreover, the classical Lie algebras
and Zassenhaus algebras can be generated by 1.5 elements, that is,
 any given nonzero element can be paired a suitable element such that these two elements
  generate  the whole algebra.   The results in \cite{MBJ} cover certain classical results in some previous papers:  In 1976, Ionescu
\cite{GIT} proved that a simple Lie algebra $L$ over $\mathbb{C}$ can be generated by $1.5$ elements and  in 1951, Kuranashi
\cite{GKM} proved that a semi-simple Lie algebra in characteristic 0
can be generated by 2 elements.  Later, Bois \cite{MBJ1} proved that the
simple graded Lie algebras of Cartan type excluding Zassenhaus
algebras are never generated by 1.5 elements. In 2009, we began to
consider the problem of generators for simple Lie superalgebras over a field of characteristic $0$ and
obtained that
 such a simple Lie superalgebra can be generated by 1 element (see arXiv:1204.2398v1
[math-ph]). We should also mention a well-known result in finite
groups: any finite simple group can be generated by 2 elements and
only the ones of prime order can be generated by 1 element
\cite{GK}.

   Let us point out the main difficulties one will encounter in simple modular Lie superalgebra case,
   which do not occur in characteristic zero situation.  In the characteristic $p$ case, for a simple  $\mathbb{Z}$-graded Lie superalgebra
   of Cartan type $X=\oplus X_{i=-l}^{s}$,
   its $\mathbb{Z}$-grading is not necessarily  consistent
with its $\mathbb{Z}_{2}$-grading. In particular, its null $X_{0}$
is not necessarily a Lie algebra. Moreover, a simple
$\mathbb{Z}$-graded modular Lie superalgebra of Cartan type  is not
necessarily generated by its local part.
In view of certain general properties of simple Lie
superalgebras,
 we overcome these issues by virtue of the weight space
 decompositions
of $X_{-1}$, $X_{0}$ and $X_{s}$  relative to the standard Cartan
subalgebra of
 $X_{0}$ and find the desirable
generators from  the odd weight vectors corresponding to
different odd weights.

 Throughout we write $\langle Y\rangle$ for
the sub-Lie superalgebra generated by a subset $Y$ in a Lie
superalgebra. The ground field $\mathbb{F}$ is algebraically closed and
of
 characteristic $p> 3$. All the vector spaces including
algebras and modules are finite dimensional.

\section{ Classical  Lie superalgebras}
In this section  $L$ denotes one of the following classical  Lie
superalgebras and their relatives: $\mathfrak{gl}(m,n)$,
$\mathfrak{sl}(m,n),$ $\mathfrak{osp}(m,n),$ ${\rm{P}}(m)$,
$\widetilde{{\rm{P}}}(m),$ $\widetilde{{\rm{P}}}(m)\oplus
\mathbb{F}I$ or $\mathfrak{osp}(m,n)\oplus \mathbb{F}I$. The aim of
this section is to prove that $L$ can be generated by 1 element
except $\mathfrak{gl}(m,n)$ with $m-n\equiv0\pmod{p},$
${\widetilde{\rm{P}}(m)}$ with $ m\not\equiv0\pmod{p}$ and
${\widetilde{\rm{P}}(m)}\oplus \mathbb{F}I$ with $
m\not\equiv0\pmod{p}$ and in these exceptional cases, $L$ can be
generated by 2 elements. The results will be used in the next
section to study generators of the $\mathbb{Z}$-graded simple Lie
superalgebras of Cartan type. Note that  the $\mathbb{Z}$-nulls of
these graded Lie superalgebras are just the classical Lie
superalgebras or their relatives mentioned above.

The general linear Lie superalgebra $\frak{gl}(m,n)$ contains the following subalgebras:
$$\frak{sl}(m,n):=\{A\in \frak{gl}(m,n)\mid \mathrm{str}(A)=0\};$$
$$\frak{osp}(m,n):=\Bigg\{\begin{bmatrix}
A&B\\
C&D
\end {bmatrix}\in \frak{gl}(m,n)\;\Big|\; A^tG+GA=0, B^tG+MC=0,
D^tM+MD=0\Bigg\},$$ where $G=\begin{bmatrix}
0&I_r\\
-I_r&0
\end {bmatrix}, M=\begin{bmatrix}
0&I_q\\
I_q&0
\end {bmatrix}$ for $n=2q$ and $\begin{bmatrix}
1&0&0\\0&0&I_q\\
0&I_q&0
\end {bmatrix}$ for $n=2q+1;$
$$ \widetilde{{\rm{P}}}(m):=\Bigg\{\begin{bmatrix}
A&B\\
C&-A^{t}
\end{bmatrix}
 \in \frak{gl}(m,m)\;\Big|\; B=B^{t}, C=-C^{t}\Bigg\};$$
 $${{\rm{P}}}(m):=\Bigg\{\begin{bmatrix}
A&B\\
C&-A^{t}
\end{bmatrix}
 \in \widetilde{{\rm{P}}}(m)\;\Big|\; \mathrm{tr}A=0\Bigg\}.$$
 Let $\mathfrak{h}$ be the standard Cartan subalgebra of $L$.  Consider
  the weight (root)
decompositions relative to $\mathfrak{h}$:
\begin{eqnarray}
&& L_{\bar{0}}=\mathfrak{h}\oplus \bigoplus_{\alpha\in
\Delta_{\bar{0}}}L_{\bar{0}}^{\alpha},\qquad
L_{\bar{1}}=\bigoplus_{\beta\in
\Delta_{\bar{1}}}L_{\bar{1}}^{\beta};\nonumber\\
\label{eq-root-space-decompose} && L=\mathfrak{h}\oplus \bigoplus_{\alpha\in
\Delta_{\bar{0}}}L_{\bar{0}}^{\alpha}\oplus\bigoplus_{\beta\in
\Delta_{\bar{1}}}L_{\bar{1}}^{\beta}.\nonumber
\end{eqnarray}
Write
   $$
   \Delta=\Delta_{\bar{0}}\cup \Delta_{\bar{1}}\quad\mbox{and}\quad L^{\gamma}=L_{\bar{0}}^{\gamma}\oplus L_{\bar{1}}^{\gamma}\quad\mbox{for}\;\gamma\in \Delta.
   $$

\begin{lemma}\label{lem-weight-information} The following statements hold.
\begin{itemize}
\item[$\mathrm{(1)}$] If $L$ is  simple then $[L_{\bar{1}},L_{\bar{1}}]=L_{\bar{0}}.$
\item[$\mathrm{(2)}$]If $L\neq\
\mathrm{P}(4)$ then $\mathrm{dim}L^{\gamma}=1$ for every $\gamma\in
\Delta.$
\end{itemize}
\end{lemma}
\begin{proof}
(1) The proof is standard (see \cite[Proposition 1.2.7(1),
p.20]{Kac}).

 (2) First, suppose $L={\rm{P}}(m)$ with $m\neq4$ and $m\not \equiv0\pmod p$. Then ${\rm{P}}(m)$ has a so-called standard basis:
$$
\begin{tabular}{|l|l|}
\multicolumn{2}{c}{ Table 1.1} \\[1pt]
\hline &$\mathfrak{h}$ \;\vline~~~ $e_{11}-e_{i+1,i+1}-e_{m+1,m+1}+e_{m+i+1,m+i+1},\,\ i\in \overline{1,m-1}$\\
 \cline{2-2}  \raisebox{1.6ex}[0pt]{$L_{\bar{0}}$} &~~~~~~~$e_{ij}-e_{m+j,m+i},\,\ i,j\in \mathbf{Y_{0}},i\neq j$\\
 \hline  $L_{\bar{1}}$&~~$e_{i,m+j}+e_{j,m+i},e_{m+i,j}-e_{m+j,i},\,\ i,j\in \mathbf{Y_{0}}$\\
  \hline
\end{tabular}
$$
Let  $\varepsilon_{i}$ be the linear function  on $\mathfrak{h}$ such that
$$\varepsilon_{i}(e_{11}-e_{j+1,j+1}-e_{m+1,m+1}+e_{m+1+j,m+1+j})=\delta_{ij}$$
for $i,j\in \overline{1,m-1}$.  Write
$1=\sum_{i=1}^{m-1}\varepsilon_{i}$. All the  weights  and the
corresponding weight vectors are as follows:
$$
\begin{tabular}{|l|l|}
\multicolumn{2}{c}{ Table 1.2} \\[1pt]
  \hline
  even weights &$ \delta_{i1}-\delta_{j1}-\varepsilon_{i-1}+\varepsilon_{j-1},\;  i,j\in \mathbf{Y_{0}},i\neq j$\\
  \hline
 weight vectors &$e_{ij}-e_{m+j,m+i},\,\  i,j\in \mathbf{Y_{0}},i\neq j$\\\hline
 odd  weights &$ \delta_{i1}+\delta_{j1}-\varepsilon_{i-1}-\varepsilon_{j-1},\;i, j\in \mathbf{Y_{0}}$\\
  \hline
weight vectors &$e_{i,m+j}+e_{j,m+i},\; i,j\in \mathbf{Y_{0}} $   \\
  \hline
 odd weights &$\delta_{j1}-\delta_{i1}+\varepsilon_{i-1}+\varepsilon_{j-1},\;  i,j\in \mathbf{Y_{0}},i\neq j$\\
  \hline
weight vectors &$e_{m+i,j}-e_{m+j,i},\,\  i,j\in \mathbf{Y_{0}},i\neq j$   \\
 \hline
\end{tabular}\\
$$
According to Table 1.2, the weight spaces of ${\rm{P}}(m)$  are all 1-dimensional.

Secondly, suppose $L={\rm{P}}(m)$ with $m\neq 4$ and $m\equiv 0\pmod
p$.  Then, all basis elements of the Cartan subalgebra
$\mathfrak{h}$ in Table
  1.1
  and the element $\sum_{i=1}^{m}e_{ii}-\sum_{j=m+1}^{2m}e_{jj}$ constitute a basis of the standard Cartan
  subalgebra of $L$. Then the conclusion
follows
  from a case-by-case examination.

  Thirdly, for $L\neq {\rm{P}}(m),$ the proof follows from a computation of weights (see \cite{Kac,Z}).
\end{proof}

For later use we record two basic facts:
\begin{enumerate}
\item[F1.] Let $V$ be a vector space and
$\mathfrak{F}=\{f_{1},\ldots,f_{n}\}$ a finite set of linear
functions on $V$. Then
$\Omega_{\mathfrak{F}}:=\{v\in V\mid \Pi_{1\leq i\neq j\leq
n}(f_{i}-f_{j})(v)\neq 0\}
 \neq\emptyset$  (see \cite[Lemma
2.2.1]{MBJ}\label{lemma-zarisk}).

  \item[F2.]
 Let $\mathfrak{A}$ be an algebra. For $a\in \mathfrak{A}$ write
$L_a$ for the left-multiplication operator given by $a$. Suppose
$x=x_{1}+x_{2}+\cdots+x_{n}$ is a sum of eigenvectors of $L_a$
associated with mutually distinct eigenvalues. Then all $x_{i}$'s
lie in the subalgebra generated by $a$ and $x.$
\end{enumerate}

\begin{proposition}\label{generators of null}
$L$  can be generated by 1 element except that
$L=\mathfrak{gl}(m,n)$ with $m-n\equiv0 \pmod p,$
${\widetilde{\rm{P}}(m)}$ with $m\not\equiv0 \pmod p$ or
${\widetilde{\rm{P}}(m)}\oplus \mathbb{F}I$ with $ m\not\equiv0
\pmod p$. In the exceptional case, $L$ can be generated by 2
elements.
\end{proposition}
\begin{proof}
  We treat two cases
separately:\\

  \noindent \textit{Case 1}.  Suppose
  $L=\mathfrak{sl}(m,n)$ with $m-n\not \equiv 0\pmod p,$
$\mathfrak{osp}(m,n)$ or ${\rm{P}}(m)$ with $ m\not \equiv 0\pmod
p.$ If $L\neq\mathrm{P}(4)$, then all the weights spaces are
$1$-dimensional. By the basic fact F1, choose any $h\in
\Omega_{\Delta_{\bar{1}}}\subset H$. Let
$x=\sum_{\gamma\in\Delta_{\bar{1}}}x_{\bar{1}}^{\gamma},$ where
$x_{\bar{1}}^{\gamma}$ is a weight vector of $\gamma$.
   By the basic fact F2 and Lemma
\ref{lem-weight-information}, we have $\langle h+x\rangle=L$.

Suppose $L={\rm{P(4)}}.$ Then the standard  basis of ${\rm{P(4)}}$
is listed above (see Table 1.1 for $m=4$). Let $x$ be the sum of all
standard odd weight vectors (basis elements) (see Table 1.2 for
$m=4$). Choose an element $h\in \Omega_{\Delta_{\bar{1}}}.$ We  want
to show that $\langle h+x\rangle=L.$

Recall that $\varepsilon_{i}$ is the linear function on
$\mathfrak{h}$ given by
$$\varepsilon_{i}(e_{11}-e_{1+j,1+j}-e_{55}+e_{5+j,5+j})=\delta_{ij}$$
for $1\leq i,j\leq 3.$
 All the odd weights and the corresponding odd weight vectors are listed below:
$$
\begin{tabular}{|c|c|c|c|c|}
\multicolumn{5}{c}{ Table 1.3} \\[1pt]
  \hline
  weights& $2\varepsilon_{1}+2\varepsilon_{2}+2\varepsilon_{3}$&$ -2\varepsilon_{1}$ & $-2\varepsilon_{2}$ & $-2\varepsilon_{3}$  \\
  \hline
  vectors &$ e_{15}$ & $e_{26}$ & $e_{37}$&$ e_{48}$ \\
\hline
  weights & $\varepsilon_{1}+\varepsilon_{3}$&$-\varepsilon_{1}-\varepsilon_{2}$& $-\varepsilon_{1}-\varepsilon_{3}$& \\
  \hline
  vectors & $e_{17}+e_{35},$$e_{64}-e_{82}$&
  $e_{27}+e_{36},$$e_{54}-e_{81}$& $e_{28}+e_{46},$$e_{53}-e_{71}$& \\
  \hline
  weights & $\varepsilon_{2}+\varepsilon_{3}$ & $\varepsilon_{1}+\varepsilon_{2}$& $-\varepsilon_{2}-\varepsilon_{3}$&\\
\hline
 vectors
  &$e_{16}+e_{25},$$e_{74}-e_{83}$& $e_{18}+e_{45},$$e_{63}-e_{72}$&
  $e_{38}+e_{47},$$e_{52}-e_{61}$&\\
  \hline
\end{tabular}
\\
$$
Then, by  Table 1.3, one sees that $\langle h+x \rangle$ contains
the following elements
$$ e_{37},\; e_{48},\; e_{16}+e_{25}+e_{74}-e_{83},\;
e_{38}+e_{47}+e_{52}-e_{61},\; e_{28}+e_{46}+e_{53}-e_{71},$$
$$e_{15},\; e_{26},\; e_{17}+e_{35}+e_{64}-e_{82},\;
e_{27}+e_{36}+e_{54}-e_{81},\; e_{18}+e_{45}+e_{63}-e_{72}.$$
 Lie superbrackets of the odd elements
above yield that
\begin{eqnarray*}
&&e_{57}-e_{31}=[e_{37},e_{28}+e_{46}+e_{53}-e_{71}], \,\
e_{58}-e_{41}=[e_{48},e_{27}+e_{36}+e_{54}-e_{81}],\\
&&e_{56}-e_{21}=[e_{26},e_{38}+e_{47}+e_{52}-e_{61}], \,\
e_{12}-e_{65}=[e_{15},e_{38}+e_{47}+e_{52}-e_{61}],\\
&&e_{13}-e_{75}=[e_{15},e_{28}+e_{46}+e_{53}-e_{71}], \,\
e_{23}-e_{76}=[e_{26},e_{18}+e_{45}+e_{63}-e_{72}],\\
&&e_{24}-e_{86}=[e_{26},e_{17}+e_{35}+e_{64}-e_{82}], \,\
e_{14}-e_{85}=[e_{15},e_{27}+e_{36}+e_{54}-e_{81}],\\
&&e_{34}-e_{87}=[e_{37},e_{16}+e_{25}+e_{74}-e_{83}], \,\
e_{67}-e_{32}=[e_{37},e_{18}+e_{45}+e_{63}-e_{72}],\\
&&e_{68}-e_{42}=[e_{48},e_{17}+e_{35}+e_{64}-e_{82}], \,\
e_{78}-e_{43}=[e_{48},e_{16}+e_{25}+e_{74}-e_{83}].
\end{eqnarray*}
So, by Table 1.2 we have $L_{\bar{0}}\subset \langle h+x \rangle.$
 Since
\begin{eqnarray*}&&[e_{17}+e_{35}+e_{64}-e_{82},
e_{78}-e_{43}]=e_{18}+e_{45}-e_{63}+e_{72},\\
&& [e_{27}+e_{36}+e_{54}-e_{81},
e_{78}-e_{43}]=e_{28}+e_{46}-e_{53}+e_{71},\\
&& [e_{18}+e_{45}+e_{63}-e_{72},
e_{86}-e_{24}]=e_{16}+e_{25}-e_{83}+e_{74},\\
&& [e_{28}+e_{46}+e_{53}-e_{71},
e_{67}-e_{32}]=e_{38}+e_{47}-e_{52}+e_{61},\\
&& [e_{16}+e_{25}+e_{74}-e_{83},
e_{67}-e_{32}]=e_{17}+e_{35}-e_{64}+e_{82},\\
&& [e_{38}+e_{47}+e_{52}-e_{61},
e_{86}-e_{24}]=e_{36}+e_{27}-e_{54}+e_{81},
\end{eqnarray*}
 by  a direct computation, one sees that all standard odd basis elements (see Table 1.1) lie in  $L$. Therefore,
 we have $\langle h+x\rangle=L.$
\\

  \noindent \textit{Case 2}. Suppose $L=\mathfrak{gl}(m,n)$
with $m-n\not\equiv 0\pmod p$ or
 $\mathfrak{osp}(m,n)\oplus \mathbb{F}I.$   Let us add a non-zero central element $I$ of $L$
 to $h+x$, where $h\in
\Omega_{\Delta_{\bar{1}}}\cap [L_{\bar{1}},L_{\bar{1}}]$ and $x$ is
the sum of all standard odd weight vectors (basis elements). Assert
that $\langle h+I+x\rangle=L.$ In fact, since
$\mathrm{ad}(h+I)(x)=\mathrm{ad}h(x).$  by the basic fact  F2 and
Lemma \ref{lem-weight-information}(2), we have $L_{\bar{1}}\subset
\langle h+I+x\rangle.$ Furthermore,
$[L_{\bar{1}},L_{\bar{1}}]\subset \langle h+I+x\rangle$ and then
$h\in \langle h+I+x\rangle.$ Thus $I\in \langle h+I+x\rangle$ and
hence
 $\langle
h+I+x\rangle=L.$

Suppose $ L=\mathfrak{sl}(m,n)$ with $m-n\equiv 0\pmod p$. A direct verification shows that
$[L_{\bar{1}},L_{\bar{1}}]=L_{\bar{0}}.$
 Furthermore, by the basic facts F1 and F2 and
Lemma \ref{lem-weight-information}(2), we obtain that $\langle
h+x\rangle=L,$ where $h\in \Omega_{\Delta_{\bar{1}}}$ and $x$ is the
sum of all standard odd weight vectors (basis elements) of $L$.

 Suppose
$L=\mathfrak{gl}(m,n)$ with $m-n\equiv 0\pmod p.$ Then
$L_{\bar{1}}=\mathfrak{sl}(m,n)_{\bar{1}}$ and
$[L_{\bar{1}},L_{\bar{1}}]=\mathfrak{sl}(m,n)_{\bar{0}}$. Choose an
element $e_{11}\in L$ and let $x$ be the sum of all standard odd
weight vectors (basis elements)  and $h\in
\Omega_{\Delta_{\bar{1}}}$. Then $\langle h+x,e_{11}\rangle=L$.

 Suppose  $L={\rm{P}}(m)$ with $m\equiv0 \pmod p.$ Then, the basis elements of the standard Cartan subalgebra of ${\rm{P}}(m)$ with $m\not\equiv0 \pmod
 p$ in Table
  1.1
  and the element $w:=\sum_{i=1}^{m}e_{ii}-\sum_{j=m+1}^{2m}e_{jj}$ constitute a basis of the standard Cartan
  subalgebra of $L$. Let $x$  be the sum of all standard odd weight
vectors (see  Table 1.1) and $h\in \Omega_{\Delta_{\bar{1}}}$. Then,
by the basic facts  F1 and F2 and Lemma
\ref{lem-weight-information}(2), we may deduce  that $\langle
h+x\rangle=L.$

 Suppose $L={\widetilde{\rm{P}}(m)}$ with $m\equiv0 \pmod p.$ By the basic facts F1 and F2 and Lemma \ref{lem-weight-information}(2), we have $\langle h+x\rangle=L$, where $h\in
 \Omega_{\Delta_{\bar{1}}}$ and $x$ is the sum of all standard odd weight
vectors (basis elements).

 Suppose $L={\widetilde{\rm{P}}(m)}\oplus \mathbb{F}I$ with $m\equiv0 \pmod p.$ We have $\langle h+x+I\rangle=L$, where $h\in
 \Omega_{\Delta_{\bar{1}}}$ and $x$ is the sum of all standard odd weight
vectors (basis elements).

Suppose $L={\widetilde{\rm{P}}(m)}$ with $m\not\equiv 0\pmod p.$  We
have $\langle h+x,w\rangle=L$, where $h\in
 \Omega_{\Delta_{\bar{1}}},$ $w=\sum_{i=1}^{m}e_{ii}-\sum_{j=m+1}^{2m}e_{jj}$ and $x$ is the sum of all standard odd weight
vectors (basis elements).

Suppose $L={\widetilde{\rm{P}}(m)}\oplus \mathbb{F}I$ with
$m\not\equiv0\pmod p.$ Then we have $\langle h+x+I,w\rangle=L$,
where $h\in
 \Omega_{\Delta_{\bar{1}}}\cap[L_{\bar{1}},L_{\bar{1}}],$ $w=\sum_{i=1}^{m}e_{ii}-\sum_{j=m+1}^{2m}e_{jj}$ and $x$ is the sum of all standard odd weight
vectors (basis elements). In fact,
 $h+I,$ $x$ and $w$ lie in $\langle h+x+I,w\rangle.$ In view of the basic fact F2 and Lemma \ref{lem-weight-information}(2),
 we obtain that $L_{\bar{1}}\subset\langle h+x+I,w\rangle.$ Then $h\in [L_{\bar{1}},L_{\bar{1}}]\subset\langle h+x+I,w\rangle.$
 This yields $I\in\langle
 h+x+I,w\rangle$ and thereby $\langle h+x+I,w\rangle=L.$
\end{proof}
\section{Lie superalgebras of Cartan type}
 Fix
two $m$-tuples of positive integers $ \underline{t}:=\left(
t_{1},t_{2},\ldots,t_m\right) $ and $\pi:=\left( \pi _1,\pi
_2,\ldots,\pi _m\right)$, where $\pi _i:=p^{t_i}-1.$  Let
$\mathcal{O}(m;\underline{t})$ be the \textit{divided power algebra}
over $\mathbb{F}$ with basis $\{x^{(\alpha)}\mid\alpha\in
\mathbb{A}\}$, where
$\mathbb{A}:=\{\alpha\in\mathbb{N}^{m}\mid\alpha_{i}\leq\pi_{i}\}.$
Let $\Lambda(n)$ be the \textit{exterior superalgebra} over
$\mathbb{F}$ of $n$ variables $ x_{m+1},x_{m+2},\ldots,x_{m+n}.$
Hereafter $\overline{s,t}$ is the set of integers $s,s+1,\ldots,t.$ Let
$$\mathbb{B}=\{\langle i_{1},i_{2},\ldots,i_{k}\rangle \mid m+1\leq
i_{1}<i_{2}<\cdots<i_{k}\leq m+n,  k\in \overline{1, n}\}.$$ For
$u:=\langle i_{1},i_{2},\ldots,i_{k}\rangle \in\mathbb{B},$ set
$|u|=k$ and write $x^{u}=x_{i_{1}}x_{i_{2}} \cdots x_{i_{k}}.$ The
tensor product $\mathcal{O}(m,n;\underline{t}):=\mathcal
{O}(m;\underline{t})\otimes\Lambda(n)$
 is an associative superalgebra. Note that $\mathcal{O}(m,n;\underline{t})$ has a
standard $\mathbb{F}$-basis $\{x^{(\alpha)}x^{u}\mid
(\alpha,u)\in\mathbb{A}\times\mathbb{B}\}.$ For convenience, put
$$\mathbf{Y}_{0}=\overline{1,m};\,\ \mathbf{Y}_{1}=\overline{m+1,m+n}\,\ \mbox{and}\,\
\mathbf{Y}=\mathbf{Y}_{0}\cup \mathbf{Y}_{1}.$$ Let $\partial_{r}$
be the superderivation of $\mathcal{O}(m,n;\underline{t})$ such that
$$\partial_{r}(x^{(\alpha)})=x^{(\alpha-\epsilon_{r})}\,\ \mbox{for}\,\ r\in
\mathbf{Y}_{0}\,\ \mbox{and}\,\ \partial_{r}(x_{s})=\delta_{rs}\,\
\mbox{for}\,\ r,s\in \mathbf{Y}.$$    For $g\in
\mathcal{O}(m;\underline{t}),$ $f\in \Lambda(n),$   write $gf $ for
$ g\otimes f$. For $\epsilon_i:=( \delta_{i1}, \ldots,\delta
_{im}),$ we abbreviate $x^{(\epsilon_i)}$ by $x_i,$ $i=1,\ldots,m.$
For a $\mathbb{Z}_2$-graded vector space $V$, we denote by
$|x|=\theta$ the \textit{parity of a homogeneous element} $x\in
V_{\theta}$, $\theta\in\mathbb{Z}_2$.
 Now we list the eight series of  modular graded Lie superalgebras
of Cartan type
\cite{Fu-Zhang-Jiang,Liu-He,Liu-Yuan-1,Liu-Zhang-Wang,Zhang}, which
are subalgebras of the Lie superalgebra
$\mathrm{Der}\mathcal{O}(m,n;\underline{t})$:
\begin{description}
\item[$\mathbf{W}:$]   The \textit{generalized Witt superalgebra}   is $$
W\left(m,n;\underline{t}\right):=\mathrm{span}_{\mathbb{F}}\{f
\partial_r |  f\in
\mathcal{O}(m,n;\underline{t}),r\in\mathbf{Y}\}.$$
\end{description}
\begin{description}
    \item[$\mathbf{S}:$]
 ~~Let $\mathrm{div}: W(m,n;\underline{t})\longrightarrow
\mathcal{O}(m,n;\underline{t})$ be the  {divergence}, which is the
linear mapping such that
$\mathrm{div}(f\partial_k)=(-1)^{|\partial_k||f|}\partial_k(f)$ for
all $ k\in\mathbf{Y}.$ For $i, j\in \mathbf{Y}$, let
 ${\mathrm{D}_{ij}}(a)=(-1)^{|\partial_i||\partial_j|}\partial_i(a)\partial_j-(-1)^{(|\partial_i|+|\partial_j|)|a|}\partial_j(a)
\partial_i$ for
$a\in\mathcal{O}(m,n;\underline{t}).$

 The \textit{special
superalgebra} is
$$S(m,n;\underline{t}):=\mathrm{span}_{\mathbb{F}}\{{\mathrm{D}_{ij}}(a)\mid
a\in\mathcal{O}(m,n;\underline{t}),\, \, i, j\in \mathbf{Y} \}. $$
\end{description}
Let $m=2r$ or $2r+1$, where $r\in \mathbb{N}$. Put
\[ i'= \left\{\begin{array}{lll} i+r, &\mbox{if}\;
i\in\overline{1, r};
\\i-r, &\mbox{if}\; i\in \overline{r+1, 2r};
\\i, &\mbox{if}\; i\in \mathbf{Y}_1.
\end{array}\right.\quad \sigma(i)= \left\{\begin{array}{lll} 1, &\mbox{if}\;
i\in\overline{1, r};
\\-1, &\mbox{if}\; i\in \overline{r+1, 2r};
\\1, &\mbox{if}\; i\in \mathbf{Y}_1.
\end{array}\right.\]
\begin{description}
  \item[$\mathbf{H}:$]
~Suppose $m=2r$. Let  $\mathrm{D_{H}}(a)=\sum_{i\in
\mathbf{Y}}\sigma(i)(-1)^{|\partial_i||a|}\partial_i(a)\partial_{i'}.$
The \textit{Hamiltonian superalgebra}  $H(m,n;\underline{t})$ is the derived algebra
of  the Lie superalgebra
$$\overline{H}(m,n;\underline{t}):=\mathrm{span}_{\mathbb{F}}\{\mathrm{D_{H}}(a)\mid
a\in\mathcal{O}(m,n;\underline{t})\}. $$
 \end{description}
\begin{description}
  \item
[$\mathbf{K}:$]~~Suppose $m=2r+1$. Then
$(\mathcal{O}(m,n;\underline{t}),[\;,\;]_K)$ is a Lie superalgebra with respect to multiplication
$$ [a,b]_K= \mathrm{D_{K}}(a)\left(b\right)-2\partial_m(a)(b) $$ for
$a,b\in \mathcal{O}(m,n;\underline{t})$, where
$$\mathrm{D_{K}}(a):=\sum_{i\in \mathbf{Y}\backslash\{m\}}(-1)^{|\partial_i||a|}\Big(x_i\partial_m(a)+\sigma(i')\partial_{i'}(a)\Big)\partial_i
+\Big(2a-\sum_{i\in
\mathbf{Y}\backslash\{m\}}x_i\partial_i(a)\Big)\partial_m.$$ The
derived algebra of this Lie superalgebra is  simple,
 called the \textit{contact superalgebra}, denoted by $K(m,n;\underline{t})$.
\end{description}
In the below,  suppose $m>2$ and $n=m$ or $m+1$. Put
\[ i''= \left\{\begin{array}{lll} i+m, &\mbox{if}\;
i\in\mathbf{Y_{0}};
\\i-m, &\mbox{if}\; i\in \overline{m+1, 2m}.
\end{array}\right.\]
   We have the following simple
Lie superalgebras of Cartan type \cite{Liu-He,Liu-Zhang-Wang}.
\begin{description}
  \item
[$\mathbf{HO}:$] Let
 $\mathrm{T_H}(a)=\sum_{i\in  \mathbf{Y}}(-1)^{|\partial_i||a|}\partial_i(a)\partial_{i''}$
for $a\in \mathcal{O}(m,m;\underline{t}).$
 Then $\mathrm{T_H}$ is an odd linear mapping with kernel $\mathbb{F}\cdot 1$ and $\mathbb{Z}$-degree
 $-2.$ The \textit{odd Hamiltonian superalgebra} is
 $$ HO(m,m;\underline{t}):=\mathrm{span}_{\mathbb{F}}\{\mathrm{T_{H}}(a)\mid a\in\mathcal{O}(m,m;\underline{t})\}.$$
\end{description}
\begin{description}
  \item[$\mathbf{SHO}:$] The
\textit{special odd Hamiltonian Lie superalgebra} $SHO(m,m;\underline{t})$ is the second derived superalgebra of the Lie superalgebra
 $S(m,m;\underline{t})\cap HO(m,m;\underline{t})$.
\end{description}
  From  \cite{Fu-Zhang-Jiang,Liu-Yuan-1}  we have the
following  simple Lie superalgebras.
\begin{description}
  \item
[$\mathbf{KO}:$] The \textit{odd contact Lie superalgebra} is
$$KO(m,m+1;\underline{t}):=(\mathcal{O}(m,m+1;\underline{t}),[\;,\;]_{KO})$$
with multiplication
\begin{equation*}
[a,b]_{KO}=
\mathrm{D_{KO}}\left(a\right)\left(b\right)-(-1)^{|a|}2\partial_{2m+1}\left(a\right)b\quad
\mbox{for}\; a,b\in \mathcal{O}(m,m+1;\underline{t}),
\end{equation*}
where
$$\mathrm{D_{KO}}(a):
=\mathrm{T_{H}}(a)+(-1)^{|a|}\partial_{2m+1}(a)\mathfrak{D}+
\big(\mathfrak{D}(a)-2a \big)\partial_{2m+1}$$ and
$\mathfrak{D}:=\sum _{i=1}^{2m}x_{i}\partial_{i}. $
 Note that $\mathfrak{D}$ is just the degree superderivation of
 $\mathcal{O}(m,m;\underline{t})$.
 Here,
 both $\mathrm{T_{H}}$ and $\mathfrak{D}$ are naturally extended to $\mathcal{O}(m,m+1;\underline{t})$.
\end{description}
\begin{description}
  \item
[$\mathbf{SKO}:$]
 Given $\lambda\in \mathbb{F},$ for $a\in \mathcal{O}(m,m+1;\underline{t})$, let
$$
\mathrm{div}_{\lambda}(a)=(-1)^{|a|}2\left(\sum_{i=1}^{m}\partial_i\partial_{i''}
\left(a\right)+\left(\mathfrak{D}-m\lambda\mathrm{id}_{\mathcal{O}(m,m+1;\underline{t})}\right)
\partial_{2m+1}\left(a\right)\right).$$

The kernel of $\mathrm{div}_{\lambda}$ is a subalgebra of
$KO(m,m+1;\underline{t})$. Its second derived algebra is simple,
called the \textit{special odd contact Lie superalgebra}, denoted by $SKO(m,m+1;\underline{t}).$
\end{description}

If
$V=\bigoplus_{r\in\mathbb{Z}}V_r$ is a $\mathbb{Z}$-graded vector
space and $x \in  V$ is a $\mathbb{Z}$-homogeneous element, write
$\mathrm{zd}(x)$  for the $\mathbb{Z}$-degree of $x$. From now on   $X$ denotes one of the simple graded Lie superalgebras
$W$, $S$, $H$, $K$, $HO$, $SHO$, $KO$ or $SKO$.  The
$\mathbb{Z}$-grading of $X$ given by
$\mathrm{zd}(x_i)=-\mathrm{zd}(\partial_i)=a_i$, where
$a_i=1+\delta_{X=K}\delta_{i=m}+\delta_{X=KO}\delta_{i=2m+1}+\delta_{X=SKO}\delta_{i=2m+1}$,
induces a $\mathbb{Z}$-grading  $X=X_{-l}\oplus \cdots \oplus
X_{s}$. Note that $l=1$ for $X=W, S, H,
HO$ or $SHO$ and  $l=2$ for $X=K,KO$ or $SKO.$ Put
$\theta=\sum_{i=1}^mp^{t_i}-m+n.$ Then
\[s = \left\{\begin{array}{lll} \theta-1, &X=W, \, KO;\\
\theta-2,&X=S, \, HO \;\mbox {or}\; SKO\;\mbox{with}\; m\lambda+1\not \equiv 0 \pmod p; \\
\theta-3, &X=H, \,  SKO  \;\mbox{with}\;  m\lambda+1\equiv 0 \pmod p; \\
\theta-5, &X=SHO;\\
\theta+\pi_{m}-2, &X=K  \;\mbox{with}\; n-m-3\not\equiv 0\pmod p; \\
\theta+\pi_{m}-3, &X=K  \;\mbox{with}\; n-m-3\equiv 0 \pmod p.
\end{array}\right.\]
 For $X=W$, $S$, $HO$, $SHO$, $KO$ or $SKO$, the null $X_{0}$ is
isomorphic to a classical modular Lie superalgebra or its
relatives under the canonical isomorphism
$$\phi:W(m,n;\underline{t})_{0}\longrightarrow \mathfrak{gl}(m,n)$$ given by
$x_{i}D_{j}\longmapsto e_{ij},$ where $ i,j\in \mathbf{Y}.$
For $X=H$,
as in \cite[p.164]{Zhang-Liu}, put
$$ \mathfrak{L}(m,n)=\bigg\{\begin{bmatrix}
A&B\\
C&D
\end {bmatrix}\in \frak{gl}(m,n)\;\Big{|}\; A^tG+GA=0, B^tG+C=0, D^t+D=0\bigg\}.$$
There is an isomorphism of algebras:
 $$\varphi:H(m,n;\underline{t})_{0}\longrightarrow
\mathfrak{L}(m,n)$$
 given by $${\rm{D_{H}}}(x_{i}x_{j})\longmapsto
\sigma(j)(-1)^{|\partial_{j}|}e_{ij'}+\sigma(i)\sigma(j)(-1)^{|\partial_{i}||\partial_{j}|+|\partial_{i}|+|\partial_{j}|}e_{ji'}\quad\mbox{for $i,j\in \mathbf{Y}.$}$$
For $K$, there is an isomorphism of algebras:
$$\psi:K(m,n;\underline{t})_{0}\longrightarrow \mathfrak{L}(m-1,n)\oplus
\mathbb{F}I$$  given by ${\rm{D_{K}}}(x_{m})\longmapsto I$ and for
$i,j\in \overline{1,m-1}\cup \mathbf{Y_{1}}$,
\begin{eqnarray*}
{\rm{D_{K}}}(x_{i}x_{j})\longmapsto
\sigma(j)(-1)^{|\partial_{j}|}e_{ij'}+\sigma(i)\sigma(j)(-1)^{|\partial_{i}||\partial_{j}|+|\partial_{i}|+|\partial_{j}|}e_{ji'}.
\end{eqnarray*}  Suppose $\mu\in\mathbb{F}$ and $\mu^2=-1$. Put
$$P_n=\begin{bmatrix}
I_q&\frac{1}{2}I_q\\
-\mu I_q&\frac{\mu}{2}I_q
\end {bmatrix} \;\mbox{if}\; n=2q \;\;\mbox{and}\; \begin{bmatrix}
1&0&0\\0& I_q&\frac{1}{2}I_q\\
0&-\mu I_{q}&\frac{\mu}{2}I_q
\end {bmatrix}\;\mbox{if}\; n=2q+1.$$
 Then
$$\frak{osp}(m,n)=\{P^{-1}EP\mid E\in \frak{L}(m,n)\},\;
 \; P:=\begin{bmatrix}
I_m&0\\
0&P_n
\end {bmatrix}.$$
It follows that $\frak{L}(m,n)\simeq \frak{osp}(m,n).$
 We write down
the following simple facts:
\begin{eqnarray*}&&W(m,n;\underline{t})_{0}\simeq\frak{gl}(m,n),\;\;\;\;
S(m,n;\underline{t})_{0}\simeq \frak{sl}(m,n),\\
&&HO(m,m;\underline{t})_{0}\simeq \widetilde{{\rm{P}}}(m),\;\;\;\;\;
 KO(m,m+1;\underline{t})_{0}\simeq \widetilde{{\rm{P}}}(m)\oplus \mathbb{F}I,\\
 &&SHO(m,m;\underline{t})_{0}\simeq {{\rm{P}}}(m),\;\;
 SKO(m,m+1;\underline{t})_{0}\simeq \widetilde{{\rm{P}}}(m),\\
 &&H(m,n;\underline{t})_{0}\simeq\frak{osp}(m,n),\;\;
K(m,n;\underline{t})_{0} \simeq\frak{osp}(m-1,n)\oplus\mathbb{F}I.
\end{eqnarray*}
For convenience, we list the so-called standard basis of the Cartan
subalgebra $\frak{h}_{X_0} $ of $X_{0}:$
$$\begin{tabular}{|l|l|}
\multicolumn{2}{c}{ Table 2.1} \\[1pt]
\hline
\multicolumn{1}{|c|}{$X$}&\multicolumn{1}{|c|}{ basis of $\frak{h}_{X_0}$}\\
\hline
 $W(m,n;\underline{t})$&
$x_{i}\partial_{i},$ \,\ $i\in \mathbf{Y}$\\
\hline $S(m,n;\underline{t})$&$-\mathrm{D}_{1i}(x_{1}x_{i}),$
$\mathrm{D}_{1j}(x_{1}x_{j}),$
 \,\ $ i\in \overline{2,m},$ $j\in \mathbf{Y}$\\
\hline
$H(2r,n;\underline{t})$&$\mathrm{D_{H}}(x_{i}x_{i'}+x_{j}x_{j+\lfloor
\frac{n}{2}\rfloor}),$
 \,\ $i\in \overline{1,r},$ $j\in \overline{2r+1,2r+\lfloor\frac{n}{2}\rfloor}$\\
\hline
$K(2r+1,n;\underline{t})$&$x_{i}x_{i'}+x_{j}x_{j+\lfloor
\frac{n}{2}\rfloor},$
 $x_{m},$ \,\
$i\in \overline{1,r},$ $j\in \overline{2r+2, 2r+1+\lfloor\frac{n}{2}\rfloor}$\\
\hline $HO(m,m;\underline{t})$&$\mathrm{T_{H}}(x_{i}x_{i'}),$ \,\
$i\in \mathbf{Y_{0}}$ \\
\hline
$SHO(m,m;\underline{t})$&$\mathrm{T_{H}}(x_{1}x_{1'}-x_{i}x_{i'}),$
\,\
$i\in \overline{2,m}$ \\
\hline $KO(m,m+1;\underline{t})$&$x_{i}x_{i'},$ $x_{2m+1},$ \,\
$i\in \mathbf{Y_{0}}$ \\
\hline $SKO(m,m+1;\underline{t})$&$x_{2m+1}+m\lambda x_{i}x_{i'},$
\,\
$i\in \mathbf{Y_{0}}$ \\
\hline
\end{tabular}\\
~~~
$$
The weight decomposition of the component $X_k$ relative to
  the standard Cartan subalgebra $\mathfrak{h}_{X_0}$ is:
$$X_{k}=\delta_{k,0}\mathfrak{h}_{X_0}\oplus_{\alpha\in
\Delta_{k}}X_{k}^{\alpha},\,\ \mbox{where}\,\ k\in
\overline{-l,s}.$$
 Write $(\Delta_{i})_{\bar{1}}$ for the subset  of  odd
 weights
 in $\Delta_{i},$ $i=-1,0,s-1$ or $s.$
\begin{lemma}\label{root information}
Let $X$ be one of eight series of Lie superalgebras of Cartan type.
If
$$\mbox{$X=HO(m,m;\underline{t})$ or $KO(m,m+1;\underline{t})$,}$$
where  $m$ is odd, then there exist  odd
weights $\alpha_{-1}\in (\Delta_{-1})_{\bar{1}} $ and
$\alpha_{s-1}\in (\Delta_{s-1})_{\bar{1}}$ such that $\alpha_{-1}$,
$\alpha_{s-1}$ and any $\alpha\in(\Delta_{0})_{\bar{1}}$ are
pairwise distinct.

Otherwise, there exist  odd
weights $\alpha_{-1}\in (\Delta_{-1})_{\bar{1}} $ and $\alpha_{s}\in
(\Delta_{s})_{\bar{1}}$ such that $\alpha_{-1}$, $\alpha_{s}$ and
any $\alpha\in(\Delta_{0})_{\bar{1}}$ are pairwise distinct.
\end{lemma}
\begin{proof}
For $W(m,n;\underline{t}),$ let $\zeta_{i}$ be   the linear function
on $\mathfrak{h}_{W_0}$ given by
$$\zeta_{i}(x_{j}\partial_j)=\delta_{ij}, \,\  i,j\in \mathbf{Y}.$$ We have
\begin{eqnarray*}
&&(\Delta_{-1})_{\bar{1}}=\{-\zeta_{k}\mid  k\in \mathbf{Y_{1}}
\};\\
&&(\Delta_{0})_{\bar{1}}=\{\pm(\zeta_{k}-\zeta_{l})\mid k\in \mathbf{Y_{0}}, l\in \mathbf{Y_{1}}\};\\
&& (\Delta_{s})_{\bar{1}}=\left\{\begin{array}{ll}\Big\{-\sum_{l\in
\mathbf{Y}_{0}}\zeta_{l}+\sum_{k\in
\mathbf{Y}_{1}}\zeta_{k}-\zeta_{m} \;\big{|} \;m\in
\mathbf{Y_{0}}\Big\}\; & \mbox{if}\; n\; \mbox{is odd},\\
\Big\{-\sum_{l\in \mathbf{Y}_{0}}\zeta_{l}+\sum_{k\in
\mathbf{Y}_{1}}\zeta_{k}-\zeta_{n} \;\big{|} \;  n\in
\mathbf{Y_{1}}\Big\}\; & \mbox{if}\; n\; \mbox{is
even}.\end{array}\right.
\end{eqnarray*}

For $S(m,n;\underline{t}),$ let $\eta_{i}$ be   the linear function
on $\mathfrak{h}_{S_0}$ given by
$$\eta_{i}(x_1\partial_1-x_j\partial_j)=\delta_{ij}, \;
i,j\in \overline{2,m+n}$$ and write
$\eta_{1}=\sum_{l=2}^{m+n}\eta_{l}.$ We have
\begin{eqnarray*}
&&(\Delta_{-1})_{\bar{1}}=\{\eta_{k}\mid k\in \mathbf{Y_{1}}\};\\
&&(\Delta_{0})_{\bar{1}}=\{\pm(\eta_{k}-\eta_{l})\mid k\in
\mathbf{Y_{0}}, l\in \mathbf{Y_{1}}\};\\
&&(\Delta_{s})_{\bar{1}}=\Big\{-\sum_{j\in
\mathbf{Y}_{0}}\eta_{j}+\sum_{l\in
\mathbf{Y}_{1}}\eta_{l}-\eta_{i}-\eta_{k}\;\Big | \;i+k+n \;\mbox{is
 an odd number}\Big\}.
\end{eqnarray*}
For $H(2r,2q;\underline{t}),$ let $\vartheta_{i}$ be   the linear
function on $\mathfrak{h}_{H_0}$ given by
$$\vartheta_{i}(\mathrm{D_{H}}(x_{j}x_{j'}+x_{k}x_{k+q}))=\delta_{ij}+\delta_{ik},  i, j\in
\overline{1,r}; i,k\in \overline{2r+1,2r+q}.$$ We have
\begin{eqnarray*}
&&(\Delta_{-1})_{\bar{1}}=\{\pm\vartheta_{l}\mid  l\in
\overline{2r+1,
2r+q}\};\\
&&(\Delta_{0})_{\bar{1}}=\{\pm(\vartheta_{i}-\vartheta_{j})\mid
i\in \overline{1,r},\;j\in \overline{2r+1,2r+q}\};\\
&&(\Delta_{s})_{\bar{1}}=\{\pm\vartheta_{j}\mid j\in
\overline{1,r}\}.
\end{eqnarray*}
For $H(2r,2q+1;\underline{t}),$ let $\iota_{i}$ be   the linear
function on $\mathfrak{h}_{H_0}$ given by
$$\iota_{i}(\mathrm{D_{H}}(x_{j}x_{j'}+x_{k}x_{k+q}))=\delta_{ij}+\delta_{ik}, i, j\in
\overline{1,r};
i,k\in \overline{2r+1,2r+q}.$$ We have
\begin{eqnarray*}
&&(\Delta_{-1})_{\bar{1}}=\{\pm\iota_{l}\mid l\in \overline{2r+1,
2r+q}\}\cup \{0\};\\
&&(\Delta_{0})_{\bar{1}}=\{\pm(\iota_{i}-\iota_{j}),\pm\iota_{k}\mid
i\in \overline{1,r},\;j,k\in \overline{2r+1,2r+q}\};\\
&&(\Delta_{s})_{\bar{1}}=\{\pm\iota_{j}\mid j\in \overline{1,
r}\}\cup\{0\}.
\end{eqnarray*}
For $K(2r+1,2q;\underline{t}),$ let $\kappa_{i}$ be   the linear
function on $\mathfrak{h}_{K_0}$ given by
$$\kappa_{i}(x_{j}x_{j'}+x_{k}x_{k+q})=\delta_{ij}+\delta_{ik}, i, j\in \overline{1,
r};
i,k\in \overline{2r+1,2r+q}$$ and $\kappa_{m}(x_{m})=1.$ We have
\begin{eqnarray*}
&&(\Delta_{-1})_{\bar{1}}=\{\pm\kappa_{l}\mid l\in\overline{2r+2,
2r+q+1}\};\\
&&(\Delta_{0})_{\bar{1}}=\{\pm(\kappa_{i}-\kappa_{j})\mid
i\in \overline{1, r},j\in \overline{2r+2,2r+q+1}\};\\
&&(\Delta_{s})_{\bar{1}}=\left\{\begin{array}{ll}\{-4\kappa_{m}\}\;\;\;\;\;~~~~~~~~~~~~~~~~~~~~~~~~~~~\;\mbox{if}\;n-m-3\not\equiv 0 \pmod p,\\
\{-2\kappa_{j}-5\kappa_m\mid j\in \overline{1,r}\cup \overline{2r+2,
2r+q+1}\}\cup\{0\}\;\mbox{otherwise}.\end{array}\right.
\end{eqnarray*}
For $K(2r+1,2q+1;\underline{t}),$ let $\lambda_{i}$ be   the linear
function on $\mathfrak{h}_{K_0}$ given by
$$\lambda_{i}(\mathrm{D_{H}}(x_{j}x_{j'}+x_{k}x_{k+q}))=\delta_{ij}+\delta_{ik}, i, j\in
\overline{1,r};
i,k\in \overline{2r+1,2r+q}$$ and $\lambda_{m}(x_{m})=1.$ We have
\begin{eqnarray*}
&&(\Delta_{-1})_{\bar{1}}=\{\pm\lambda_{l}\mid l\in \overline{2r+2,
2r+q+1}\}\cup \{0\};\\
&&(\Delta_{0})_{\bar{1}}=\{\pm(\lambda_{i}-\lambda_{j}),\lambda_{k}\mid
i\in \overline{1,r},j,k\in\overline{2r+2,2r+q+1}\};\\
 &&(\Delta_{s})_{\bar{1}}=\left\{\begin{array}{ll}\{-4\lambda_{m}\}\;\;\;\;\;~~~~~~~~~~~~~~~~~~~~\mbox{if}\;n-m-3\not\equiv 0 \pmod p,\\
\{-2\lambda_{j}-5\lambda_m\mid j\in \overline{1, r}\cup
\overline{2r+2, 2r+q+1}\}\;\mbox{otherwise}.\end{array}\right.
\end{eqnarray*}
For $HO(m,m;\underline{t}),$ let $\mu_{i}$ be   the linear function
on $\mathfrak{h}_{HO_0}$ given by
$$\mu_{i}\big(\mathrm{T_{H}}(x_{j}x_{j'})\big)=\delta_{ij},\;
i,j\in \mathbf{Y_{0}}.$$ We have
\begin{eqnarray*}
&&(\Delta_{-1})_{\bar{1}}=\{\mu_{l}\mid l\in \mathbf{Y_{0}}\};\\
&&(\Delta_{0})_{\bar{1}}=\{2\mu_{l},\pm(\mu_{i}+\mu_{j})\mid
 i, j,l\in \mathbf{Y_{0}},i\neq j\};\\
 &&(\Delta_{s-1})_{\bar{1}}=\{-\mu_{l}\mid l\in
 \mathbf{Y_{0}}\}\cup\{0\}\;\;
 \mbox{if $m$ is odd,}\\
&&(\Delta_{s})_{\bar{1}}=\Big\{-2\sum_{i=1}^{m}\mu_{i}\Big\}\;\;
\mbox{if $m$ is even}.
\end{eqnarray*}
For $SHO(m,m;\underline{t}),$ let $\nu_{i}$ be   the linear function
on $\mathfrak{h}_{SHO_0}$ given by
$$\nu_{i}\big(\mathrm{T_{H}}(x_{1}x_{1'}-x_{j}x_{j'})\big)=\delta_{ij},\;i,j\in \overline{2,m}$$ and write $\nu_{1}=\sum_{l=2}^{m}\nu_{l}.$ We have
\begin{eqnarray*}
&&(\Delta_{-1})_{\bar{1}}=\{\nu_{i}\mid i\in \mathbf{Y_{0}}\};\\
&&(\Delta_{0})_{\bar{1}}=\{2\nu_{l},\pm(\nu_{i}+\nu_{j})\mid
i,j,l\in \mathbf{Y_{0}},i\neq j\};\\
&&(\Delta_{s})_{\bar{1}}=\{\nu_{i}\mid i\in \mathbf{Y_{0}}\}.
\end{eqnarray*}
For $KO(m,m+1;\underline{t}),$  let $\xi_{i}$ be   the linear
function on $\mathfrak{h}_{KO_0}$ given by
$$\xi_{i}(x_{1}x_{1'}-x_{j}x_{j'})=\delta_{ij},\;
i,j\in \overline{2, m}$$ and $\xi_{2m+1}(x_{2m+1})=1.$ We have
\begin{eqnarray*}
&&(\Delta_{-1})_{\bar{1}}=\{\xi_{j}+\xi_{2m+1}\mid
j\in \mathbf{Y_{0}}\};\\
&&(\Delta_{0})_{\bar{1}}=\{2\xi_{l},\pm(\xi_{i}+\xi_{j})\mid
i, j,l\in \mathbf{Y_{0}},i\neq j\};\\
&&(\Delta_{s-1})_{\bar{1}}=\{-\xi_{j}-2\xi_{2m+1}\mid j\in \mathbf{Y_{0}}\}\;\; \mbox{if $m$ is odd},\\
&&(\Delta_{s})_{\bar{1}}=\Big\{-2\big(\sum_{i=2}^{m}\xi_{i}+\xi_{2m+1}\big)-2\xi_{2m+1}\Big\}\;\;
\mbox{if $m$ is even}.
\end{eqnarray*}
For $SKO(m,m+1;\underline{t}),$  let $\omega_{i}$ be   the linear
function on $\mathfrak{h}_{SKO_0}$ given by
$$\omega_{i}(x_{1}x_{1'}-x_{j}x_{j'})=\delta_{ij},\;
i,j\in \overline{2, m}$$ and $\omega_{2m+1}(x_{2m+1})=1.$ We have
\begin{eqnarray*}
&&(\Delta_{-1})_{\bar{1}}=\Big\{\sum_{i=2}^{m}\omega_{i}+\omega_{2m+1}+\omega_{l}\;\Big{|}\; l\in \mathbf{Y_{0}}\Big\};\\
&&(\Delta_{0})_{\bar{1}}=\{2\omega_{l},\pm(\omega_{i}+\omega_{j})\mid
 i, j,l\in \mathbf{Y_{0}},i\neq j\};\\
&&(\Delta_{s})_{\bar{1}}=\left\{\begin{array}{ll}\big\{2\big(\sum_{i=2}^{m}\omega_{i}+\omega_{2m+1}\big)\big\}\;\;~\mbox{if}\;m\lambda+1\not\equiv
0 \pmod p,\\
\big\{2\big(\sum_{i=2}^{m}\omega_{i}+\omega_{2m+1}\big)-\omega_{k}\;\big{|}\;
k\in \mathbf{Y_{0}}\big\}\;\;\;\mbox{otherwise}.\end{array}\right.
\end{eqnarray*}
Summarizing, if $X=HO(m,m;\underline{t})$ with $ m$  odd or
$KO(m,m+1;\underline{t})$ with $m$ odd, we may easily find the
desired weights $\alpha_{-1}\in (\Delta_{-1})_{\bar{1}} $ and
$\alpha_{s-1}\in (\Delta_{s-1})_{\bar{1}}$ such that $\alpha_{-1}$,
$\alpha_{s-1}$ and any $\alpha\in(\Delta_{0})_{\bar{1}}$ are
pairwise distinct. Otherwise, we may easily find the desired weights
$\alpha_{-1}\in (\Delta_{-1})_{\bar{1}} $ and $\alpha_{s}\in
(\Delta_{s})_{\bar{1}}$ such that $\alpha_{-1}$, $\alpha_{s}$ and
any $\alpha\in(\Delta_{0})_{\bar{1}}$ are pairwise distinct.
\end{proof}

As in Lie algebra case  (see \cite[Propositions 3.3.5 and
3.3.6]{Strade-Farnsteiner}), we have
\begin{lemma}\label{lem-Cartan-component1}  Let $\frak{g}=\oplus_{i=-r}^{s} \frak{g}_{i}$ be a finite-dimensional simple Lie superalgebra.
Then the following statements hold.
 \begin{itemize}
\item[$\mathrm{(1)}$] $\frak{g}_{0}$-modules $\frak{g}_{-1}$ and
$\frak{g}_{s}$ are  irreducible.

\item[$\mathrm{(2)}$]  $\frak{g}$ can be generated by
$\frak{g}_{-1}$ and $\frak{g}_{s}.$
\end{itemize}
\end{lemma}

\begin{theorem}\label{th-main-cartan}
Let $X$ be one of eight series of Lie superalgebras of Cartan type.
If
\[ X= \left\{\begin{array}{lll} W(m,n;\underline{t})&\; \mbox{with}\; m-n\equiv0\pmod p,\\
HO(m,m;\underline{t})&\;\mbox{with}\; m\not\equiv0\pmod p, \\
KO(m,m+1;\underline{t}) &\;\mbox{with}\; m\not\equiv0\pmod
p,\\SKO(m,m+1;\underline{t}) &\;\mbox{with}\; m\not\equiv0\pmod p,
\end{array}\right.\]   then $X$ can
be generated by 2 elements.
Otherwise, $X$ can be generated by 1 element.
\end{theorem}
\begin{proof}
Recall that the null $X_0$ is isomorphic to one of the Lie
superalgebras
$$ \mbox{$\mathfrak{gl}(m,n),$
$\mathfrak{sl}(m,n),$ $\mathfrak{osp}(m,n),$
$\mathfrak{osp}(m-1,n)\oplus \mathbb{F}I,$ $\widetilde{\rm{P}}(m),$
$\widetilde{\rm{P}}(m)\oplus \mathbb{F}I$ or P$(m).$}$$ Let us prove
the first conclusion. From the proof of Proposition \ref{generators
of null}, $X_{0}$ can be generated by 2 elements. Write $h_{X}$ for
one of two generators which is associated with $h$ (see Proposition
\ref{generators of null},  $h+x$ for $W_{0}$, $h+x$ for $HO_{0}$,
$h+x$ for $SKO_{0}$, $h+x+I$ for $KO_{0}$) and $w_{X}$ ($e_{11}$ for
$W_{0}$, $w$ for $HO_{0}$, $w$ for $KO_{0}$, $w$ for $KO_{0}$) for
the other generator for $X_{0}$.

For $X\neq HO(m,m;\underline{t})$ and $KO(m,m+1;\underline{t})$ with
$m$  odd,  by Lemma \ref{root information}, we choose
 weights $\alpha_{-1}\in (\Delta_{-1})_{\bar{1}}$ and
$\alpha_{s}\in (\Delta_{s})_{\bar{1}}$ such that
$\alpha_{-1},\alpha_{s}$ and any $\alpha\in (\Delta_{0})_{\bar{1}}$
are pairwise distinct. Let
$\Phi=\{\alpha_{-1}\}\cup\{\alpha_{s}\}\cup (\Delta_{0})_{\bar{1}}.$
Choose any $h\in \Omega_{\Phi}\subset
\Omega_{(\Delta_{0})_{\bar{1}}}$ and odd weight vectors $ x_{-1}\in
X_{-1}^{\alpha_{-1}} $ and $ x_{s}\in X_{s}^{\alpha_{s}}.$ Let us
add $x_{-1}$ and $x_{s}$
 to $w_{X}.$ Write $x'=x_{-1}+w_{X}+x_{s}.$ Assert that $\langle h_{X}, x'\rangle=X.$ In fact,   we have
$x_{-1}+x_{s},w_{X},h_{X}\in \langle h_{X}, x'\rangle$. Furthermore,
$X_0=\langle h_{X}, w_{X}\rangle \subset \langle h_{X}, x'\rangle.$
We choose $h'\in \Omega_{\{\alpha_{-1},\alpha_{s}\}}\subset
\mathfrak{h}_{X_{0}}\subset X_{0}$. Then, by the basic fact F2, we have
$x_{-1}$, $x_{s}\in \langle h_{X}, x'\rangle\subset X.$ By Lemma
\ref{lem-Cartan-component1}(1), the irreducibility of $X_{-1}$ and
$X_{s}$ as $X_0$-modules ensures that $X_{-1}\subset X$ and
$X_{s}\subset X.$ According to Lemma \ref{lem-Cartan-component1}(2),
we obtain that $X=\langle h_{X}, x'\rangle$.

For $X=HO(m,m;\underline{t})$  or
$KO(m,m+1;\underline{t})$ with $m$  odd, by Lemma \ref{root
information} we choose
 weights $\alpha_{-1}\in (\Delta_{-1})_{\bar{1}}$ and
$\alpha_{s-1}\in (\Delta_{s-1})_{\bar{1}}$ such that
$\alpha_{-1},\alpha_{s-1}$ and any $\alpha\in
(\Delta_{0})_{\bar{1}}$ are pairwise distinct. Let
$\Phi=\{\alpha_{-1}\}\cup\{\alpha_{s-1}\}\cup
(\Delta_{0})_{\bar{1}}.$ Put $y'=x_{-1}+w_{X}+x_{s-1}$ and choose
$h\in\Omega_{\Phi}\subset \Omega_{(\Delta_{0})_{\bar{1}}}$. Then
$w_{X},$ $x_{-1}+x_{s-1}\in \langle h_{X},y'\rangle.$ From the proof
of Proposition \ref{generators of null}, we have $X_{0}=\langle
h_{X}, w_{X}\rangle\subset \langle h_{X},y' \rangle.$ Then, we
choose $h'\in \Omega_{\{\alpha_{-1},\alpha_{s}\}}\subset
\mathfrak{h}_{X_{0}}\subset X_{0}$. Then, by the basic fact F2, we
have $x_{-1}$, $x_{s}\in \langle h_{X}, y'\rangle\subset X.$ The
irreducibility of $X_{-1}$ as $X_{0}$-module ensures $X_{-1}\subset
X.$ Since $x_{s-1}\in \langle h_{X},y' \rangle$ and transitivity of
$X$, we may conclude that there exists an  element $x_{1}$ of
$X_{1}$ in $\langle h_{X},y'\rangle$ such that
$X_{s}=\mathbb{F}[x_{1},x_{s-1}]\subset \langle h_{X},y'\rangle.$
Then by Lemma \ref{lem-Cartan-component1}, we have $X=\langle
h_{X},y'\rangle.$

Next we prove the second conclusion. From the proof of Proposition
\ref{generators of null}, $X_{0}$ can be generated by $x_{0}+h,$
where $h\in \Omega_{(\Delta_{0})_{\bar{1}}}$ and $x_{0}$ is the sum
of all the standard odd weight vectors for $X=S$, $H,$ $HO,$ $SHO$
or $SKO$ and  the sum of all the odd weight vectors and certain
element $I$ for $X=W,$ $K$ or $KO$.

For $X\neq HO(m,m;\underline{t})$ and $KO(m,m+1;\underline{t})$ with
$m$ odd, by Lemma \ref{root information}, we choose
 weights $\alpha_{-1}\in (\Delta_{-1})_{\bar{1}}$ and
$\alpha_{s}\in (\Delta_{s})_{\bar{1}}$ such that
$\alpha_{-1},\alpha_{s}$ and any $\alpha\in (\Delta_{0})_{\bar{1}}$
are pairwise distinct. Put $x=x_{-1}+x_{0}+x_{s}$ for some odd
weight vectors $ x_{-1}\in X_{-1}^{\alpha_{-1}} $ and $ x_{s}\in
X_{s}^{\alpha_{s}}.$ Let $\Phi=\{\alpha_{-1}\}\cup\{\alpha_{s}\}\cup
(\Delta_{0})_{\bar{1}}.$ Choose any $h\in \Omega_{\Phi}\subset
\Omega_{(\Delta_{0})_{\bar{1}}}$. Assert that $\langle
h+x\rangle=X.$ In fact, we have $x_{-1}+x_{0}+x_{s}\in \langle
h+x\rangle$ and $h\in \langle h+x\rangle.$ Furthermore, according to
the basic fact F2, we have $x_{-1}, x_{0},x_{s}\in \langle
h+x\rangle.$ Then $\langle h+x_0\rangle =X_0\subset \langle
h+x\rangle.$  By Lemma \ref{lem-Cartan-component1}(1), the
irreducibility of $X_{-1}$ and $X_{s}$ as $X_0$-modules ensures that
$X_{-1}\subset X$ and $X_{s}\subset X.$ According to Lemma
\ref{lem-Cartan-component1}(2), we have $X=\langle h+x\rangle$.

For $X=HO(m,m;\underline{t})$ with $m$  odd or
$KO(m,m+1;\underline{t})$ with $m$  odd, by Lemma \ref{root
information} we choose
 weights $\alpha_{-1}\in (\Delta_{-1})_{\bar{1}}$ and
$\alpha_{s-1}\in (\Delta_{s-1})_{\bar{1}}$ such that
$\alpha_{-1},\alpha_{s-1}$ and any $\alpha\in
(\Delta_{0})_{\bar{1}}$ are pairwise distinct. Let
$\Phi=\{\alpha_{-1}\}\cup\{\alpha_{s-1}\}\cup
(\Delta_{0})_{\bar{1}}.$ Put $y=x_{-1}+x_{0}+x_{s-1}$ and choose
$h\in\Omega_{\Phi}\subset \Omega_{(\Delta_{0})_{\bar{1}}}$. Then
$x_{-1},$ $x_{0},$ $x_{s-1}\in \langle h+y\rangle.$ From the proof
of Proposition \ref{generators of null}, we have $X_{0}=\langle
h+x_{0}\rangle\subset \langle h+y \rangle.$ The irreducibility of
$X_{-1}$ as $X_{0}$-module ensures $X_{-1}\subset X.$ Since
$x_{s-1}\in \langle h+y \rangle$ and transitivity of $X$, we
conclude that there exists an  element $x_{1}$ of $X_{1}$ in
$\langle h+y\rangle$ such that
$X_{s}=\mathbb{F}[x_{1},x_{s-1}]\subset \langle h+y\rangle.$ Then by
Lemma \ref{lem-Cartan-component1}, we have $X=\langle h+y\rangle.$
\end{proof}

\end{document}